\documentclass[a4paper,11pt]{article}
\usepackage{amssymb,amsmath,amsthm,color}
\usepackage{graphics,graphicx,color}
\usepackage{srcltx}
\usepackage[normalem]{ulem}
\DeclareGraphicsExtensions{.eps} \textwidth=169mm
\definecolor{darkred}{rgb}{0.9,0.1,0.1}

\newtheorem{proposition}{Proposition}[section]
\newtheorem{theorem}{Theorem}[section]
\newtheorem{lemma}{Lemma}[section]
\newtheorem{corollary}{Corollary}
\newtheorem{remark}{Remark}[section]
{\rm}
\definecolor{darkred}{rgb}{0.9,0.1,0.1}

\def\eps{\varepsilon}
\def\eps{\varepsilon}
\begin{document}

\title{Homogenization of a diffusion in a high-contrast random 
 environment and related Markov semigroups}
\author{B. Amaziane$^1$, \ A. Piatnitski$^{2,3,}$\footnote{The work of the second  author was supported
by Leonhard Euler International Mathematical Institute in Saint Petersburg, grant No. 075-15-2019-1619  }
\setcounter{footnote}{3}
,\ E. Zhizhina$^{3,}$\footnote{The work of the third  author was supported
by Leonhard Euler International Mathematical Institute in Saint Petersburg, grant No. 075-15-2019-1620  }}
\date{\small $^1$ Laboratoire de Math\'ematiques et de leurs Applications de Pau - UMR CNRS 5142,\\
University of Pau and  Pays de l'Adour
Avenue de l'Universit\'e - BP 1155\\
64013 Pau, France
\\[2mm]
\small $^2$ The Arctic university of Norway, campus Narvik, \\
P.O.Box 385, 8505 Narvik, Norway\\[2mm]
$^3$ The Institute for Information Transmission Problems of RAS, \\
19, Bolshoy Karetny per. 127051,  Moscow, Russia\\[5mm]
}

\maketitle

\begin{abstract}
The goal of the paper is to describe the large time behaviour of a Markov process associated with a symmetric diffusion in a high-contrast random environment and to characterize the limit semigroup and the limit process under the diffusive scaling.
\end{abstract}

\section{Introduction}
\label{s_I}
The paper focuses on the large time behaviour of a diffusion in a high contrast random statistically homogeneous environment.
We also study the limit behaviour of the corresponding semigroups.
Equivalently, we consider  the limit behaviour of a diffusion  defined in a high contrast
environment with a random  microstructure on a finite time interval.

Elliptic and parabolic operators with high contrast rapidly oscillating periodic coefficients have been widely studied in the homogenization theory. The first rigorous results for parabolic operators of this type were obtained in \cite{MKh} and \cite{ArDoHo}.
In particular, it was shown that, under proper choice of the scaling coefficient,  the homogenized problem contains a non-local in time operator which reflects the so-called memory effect. Later on in \cite{Al}, with the help of the two-scale convergence technique, the limit problem was written as a coupled system of parabolic PDEs in the space with higher number of variables.
In the works \cite{Zh}, \cite{Zh1} high contrast problems  in domains with singular or asymptotically singular periodic geometry were considered.
 At present, there are many works in the existing mathematical literature
that describe the effective behaviour of high contrast periodic media. Under proper scaling, in parabolic problems this usually
results in the memory effect while homogenization of spectral problems leads to a  non-linear dependence on the spectral parameter.

In this paper we deal with second order divergence form operators in $\mathbb R^d$. Each such an operator is a generator of
a Markov semigroup. The corresponding Markov process  (generalized diffusion) has continuous trajectories.   However, the presence
of a non-local in temporal variable  term in the effective operator means that the limit dynamics of the coordinate process is not Markov.

The goal of this work is to equip the coordinate process with  additional components in such a way that the dynamics of the enlarged
process remains Markovian in the limit.  We show that it is sufficient to combine the coordinate process in $\mathbb R^d$ with a position of the diffusion inside the rescaled inclusions for the time intervals when the diffusion is trapped by one of the inclusions.

It is interesting to observe that, although in the original processes the additional components are  functions of the coordinate
process, in the limit  process these components are getting independent while the coordinate process becomes coupled with them.

The explicit form of the limit operator on the extended space gives us a possibility to use this operator as a
generator of an approximation dynamics for the processes in high contrast random stationary dispersive porous media.
The discrete version of such  approximation process was constructed  in the work \cite{PZh}, where we
considered a discrete diffusion in a high-contrast random environment given by a jump random walk on the lattice $\mathbb Z^d$.
The crucial step in this construction is to describe the "clock process" governing transitions
from the observable "real" space to the supplementary "astral" spaces and back.
The "clock process" is a continuous time finite Markov chain
with transition rates depending on parameters of the limit operator.


In the paper we introduce proper functional spaces, construct the limit semigroup, and prove the semigroup convergence.


In addition to proving the semigroup convergence, we study the spectrum of the generator of the
limit semigroup.
Then the semigroup convergence in $L^2$ spaces allows us to provide some information about the limit behaviour of the spectrum
of the original operators.

To our best knowledge, the questions considered in this paper  have not been studied in the existing literature.
 In the discrete framework the results on scaling limits of symmetric  random walks in a high contrast periodic environment
 were obtained in our previous work \cite{PZh}.

 Our approach essentially relies on the approximation technique developed in \cite{EK} and the technique of
 correctors in random media. In contrast with the periodic framework, the auxiliary operators used to introduce correctors
 need not be of Fredholm type in the case of random inclusions. The construction of the first corrector can be found in the existing literature, see for instance \cite{JKO}. However, when defining the higher order correctors we face additional difficulties.


\section{Problem setup}
\label{s_setup}

Let $(\Omega,\mathcal{F},\mathbf{P})$ be a standard probability space.
Consider a symmetric diffusion operator in divergence form
\begin{equation}\label{Aeps}
A^\omega_\varepsilon f (x) = \mathrm{div} \left( a^\omega_\varepsilon(x) \nabla f(x) \right),
\end{equation}
where
\begin{equation}\label{a}
a^\omega_\varepsilon(x) = \left\{
\begin{array}{l}
\mathbb{I} ,\quad x \in \mathbb R^d \setminus \varepsilon G^\omega = \big(G^\omega_\varepsilon \big)^c, \\
\varepsilon^2 \mathbb{I} ,\quad x \in \varepsilon G^\omega,
\end{array}
\right.
\end{equation}
$\mathbb{I}$ being the unit matrix.
Here $G^\omega_\varepsilon = \varepsilon G^\omega$ and $G^\omega \subset \mathbb{R}^d,\, \omega \in \Omega,$ is  a classical disperse medium.  That is $\mathbb R^d\setminus G^\omega$ is a random statistically homogeneous set  such that almost surely (a.s) it is connected and unbounded and its complement
$G^\omega \subset \mathbb{R}^d$ consists of  a countable number of uniformly  bounded simply connected
domains with uniformly Lipschitz boundary. Moreover, the distance between any two such domains admits a uniform
deterministic lower bound. The set $ G^\omega$ corresponds to the matrix blocks (inclusions), and $\mathbb R^d\setminus G^\omega \subset \mathbb{R}^d$ to the fractures system, see \cite{BMP2003}.

To be more specific in this work we consider a class of disperse media that satisfy the following additional condition:\\
there is a
{\bf finite}  collection of bounded domains in $\mathbb R^d$ such that all these domains are uniformly $C^2$ regular and a.s. any connected component of $ G^\omega$ can be obtained by a proper translation and rotation of one of these domains.
The domains are denoted by   ${\cal D}_{j}$, $j=1,\ldots,N$,
and the whole collection by ${\cal D}$,
$\mathcal{D}=
\{\mathcal{D}_j\}_{j=1}^N$ with $N\in\mathbb Z^+$.  We assume without loss of generality that each domain $\mathcal{D}_j$, $j\geq 1$,
contains the origin.

An example of such a disperse medium is
associated with a Bernoulli 
site percolation model on the lattice $\mathbb{Z}^d$ embedded in
$\mathbb{R}^d$. Let $\{ \xi_j, \; j\in \mathbb{Z}^d \}, \ \xi_j \in \{0,1\}$ be a sequence of i.i.d. random variables having the Bernoulli law: $\mathbb{P}(\xi_j =1)=p, \; \mathbb{P}(\xi_j =0)=1-p$.
We then define $\mathtt{B}_j=j+[-\frac12,\frac12]^d$, $j\in\mathbb Z^d$, and consider the set $\mathtt{G}^{1,\omega}=\bigcup\limits_{\{j\,:\,\xi_j=1\}} \,\mathtt{B}_j$. This set is a.s. a union of countable number
of bounded connected sets (components) and not more than one unbounded connected component, see \cite{Grim}.   We replace
each bounded connected component of $\mathtt{G}^{1,\omega}$ with the minimal simply connected set that contains this component. Then we choose the sets that have Lipschitz boundary and whose volumed does not exceed $M$,
where $M\geq 1$ is a given positive integer. We then smoothen these sets and obtain a finite collection $\mathcal{D}$ of the reference sets .
By the standard arguments of percolation theory,  $G^\omega$ is statistically homogeneous and ergodic.

Denote by ${\cal G}^\omega_j$ the subset of $G^\omega$ that consists of all the components which have the same geometry as $\mathcal{D}_j$, that is each such a component can be obtained by a proper rotation and translation of $\mathcal{D}_j$ in $\mathbb{R}^d$.
The connected components of ${\cal G}^\omega_j$ are denoted by $\{{\cal G}^{\omega, i}_j\}_{i\in \mathbb Z^+}$.
We also denote ${\cal G}^\omega_0=\mathbb R^d\setminus G^\omega$.

Letting
\begin{equation}\label{alpha_sup_o}
\alpha_0 = \mathbf{P} \big\{ 0 \in (G^{\omega})^c  \big\}, \quad \alpha^o_{j} =  \mathbf{P} \big\{ 0 \in {\cal G}^\omega_{j}  \big\}, \quad \alpha_0 + \sum_{\ j\geq 1} \alpha^o_j  = 1,
\end{equation}
we assume without loss of generality that  $\alpha^o_j>0$ for all $j$.  We then introduce
\begin{equation}\label{alpha}
 \alpha_{j}= |\mathcal{D}_j|^{-1} \alpha^o_{j}=  |\mathcal{D}_j|^{-1}\mathbf{P} \big\{ 0 \in {\cal G}^\omega_{j}  \big\}.
\end{equation}

\medskip
{ For each $\eps>0$ the operator  $A^\omega_\eps$ has random statistically homogeneous coefficients in $\mathbb R^d$, where the randomness is defined through the random geometry of  $G^\omega$. These operators are also called metrically transitive with respect to the unitary group of the space translations in $\mathbb R^d$.}
In $L^2(\mathbb R^d)$ we introduce a domain of $A^\omega_\eps$ by
\begin{equation}\label{DAL2}
\begin{array}{rl}
D(A^\omega_\varepsilon) = \Big\{ f \in H^1(\mathbb R^d), \; f \in H^2 ({G_\varepsilon^\omega}) \cap H^2 ({ \mathbb R^d \setminus G_\varepsilon^\omega}), \\[1.5mm]
 \varepsilon^2
 \nabla f(x) \big|_{\partial G_\varepsilon^\omega} \cdot n^+ = -
 \nabla f(x)\big|_{\partial G_\varepsilon^\omega} \cdot n^- \Big\}
\end{array}
\end{equation}
The last relation  
in (\ref{DAL2}) represents the continuity of flux $a_\varepsilon \nabla f$ through 
the boundary $\partial G_\varepsilon^\omega$. Here $n^-, n^+$ are respectively the internal and external normals on $\partial G_\varepsilon^\omega$.

\begin{remark}
Notice that 
for any function $v\in D(A^\omega_\varepsilon)$ its trace and the trace of its flux on the interface
$\partial G_\varepsilon^\omega$ is a well-defined $L^2(\partial G_\varepsilon^\omega)$ function.
\end{remark}

Then $(A^\omega_\eps, D(A^\omega_\eps))$ is almost surely a self-adjoint operator in $L^2(\mathbb R^d)$,  and for any $\lambda>0$ the operator $(\lambda-A^\omega_\eps)$ is coercive.  By the Hille-Yosida theorem, $A^\omega_\varepsilon$ is a generator of a strongly continuous, positive, contraction semigroup $T^\omega_\varepsilon (t)$  on $L^2(\mathbb R^d)$ for a.e. $\omega \in \Omega$.

\section{The limit operator.}


In this section we describe the generator of the limit Markov semigroup. Denote
$$
E = \mathbb R^d \times {\cal D}^\star, \quad \mbox{ where } \; {\cal D}^\star = \{ \star \} \cup {\cal D}.
$$
We consider functions $F$ defined on $E$  of the following vector form
\begin{equation}\label{vectorF}
F(x, \xi) = \left\{
\begin{array}{l}
f_0 (x), \quad \mbox{if } \; x \in \mathbb{R}^d, \, \xi = \star, \\
f_{j}(x, \xi), \quad \mbox{if }  x \in \mathbb R^d, \, \xi \in  {\cal D}_j,\; j= 1,\, \ldots, N,
\end{array}
\right.
\end{equation}
with $f_0 \in L^2( \mathbb R^d), \quad f_{j} \in L^2( \mathbb R^d \times {\cal D}_{j})$.
Equipped with the norm
\begin{equation}\label{F2}
\|F\|^2 = \alpha_0  \int\limits_{\mathbb R^d} f^ 2_0(x)\,dx + \sum_{j=1}^N \alpha_{ j} \int\limits_{\mathbb R^d} \int\limits_{{\cal D}_{j}} (f_{j})^2 (x,\xi)\,d\xi dx
\end{equation}
where $\alpha_0, \ \alpha_{j} $ was defined in \eqref{alpha}, $\alpha_0>0, \ \alpha_{ j}>0 $, it is a Hilbert space.
We call this Hilbert space $L^2(E, \alpha)$.

Let us consider in  $L^2(E, \alpha)$ an operator of the following form
\begin{equation}\label{A}
(A F)(x,\xi) = \left(
\begin{array}{c}
\Theta \cdot \nabla \nabla f_0 (x) + \frac{1}{\alpha_0}  \sum_{j\geq1} \alpha_{ j}  \int\limits_{\partial  {\cal D}_{j} }  \frac{\partial f_{j} (x, \xi)}{\partial n^-_{\xi}} d \sigma(\xi) \\[4mm]
\triangle_\xi f_1 (x,\xi) \\ \ldots \\
\triangle_\xi f_{N}(x, \xi)
\end{array}
\right)
\end{equation}
where a positive definite constant matrix $\Theta$ will be defined later on, $\sigma(\xi)$ is the element of the surface volume on the Lipshitz boundary $\partial  {\cal D}_{j}^{k}$, $ n_{\xi}^-$ is the (inner) normal  to $\partial  {\cal D}_{j}^{k}$.
Using the relation $n^+ = - n^-$ and the Stokes formula one can rewrite the operator \eqref{A} as follows:
\begin{equation}\label{A-}
(A F)(x,\xi) = \left(
\begin{array}{c}
\Theta \cdot \nabla \nabla f_0 (x) - \frac{1}{\alpha_0}  \sum\limits_{j\geq1}  \alpha_{ j} \int\limits_{ {\cal D}_{j} }  \triangle_\xi f_{j} (x, \xi) d \xi \\[4mm]
\triangle_\xi f_1 (x,\xi) \\  \ldots \\
\triangle_\xi f_{N}(x, \xi)
\end{array}
\right)
\end{equation}
We denote
\begin{equation}\label{Y}
\Upsilon(x) = - \frac{1}{\alpha_0}  \sum_{j\geq1} \alpha_{ j} \int\limits_{ {\cal D}_{j} }  \triangle_\xi f_{j} (x, \xi) d \xi.
\end{equation}

For each set ${\cal D}_{j}$, $j\geq 1$, denote by $\mathbb D_j(\Delta)$ the domain of a self-adjoint operator in $L^2({\cal D}_{j})$
that corresponds to the Laplacian in   ${\cal D}_{j}$ with homogeneous Dirichlet boundary conditions. Since the boundary of ${\cal D}_{j}$ is $C^2$ regular, we have  $\mathbb D_j(\Delta)=H^2({\cal D}_{j})\cap H^1_0({\cal D}_{j})$.
Notice that this operator is positive.
The space $\mathbb D_j(\Delta)$ is equipped with the norm $\|g\|_{\mathbb D_j(\Delta)}=\|\Delta g\|_{L^2({\cal D}_{j})}$.

Defining an operator $A$ in $L^2(E, \alpha)$  by formulas \eqref{A}, \eqref{A-}, one can easily check that, with a domain
\begin{equation}\label{DAL2eff}
\begin{array}{c}
\hat D(A) = \left\{f_0 \in H^2(\mathbb R^d),\ f_j-f_0\in L^2(\mathbb R^d; {\mathbb D}_{j}(\Delta)),\    f_{j}(x,\xi)\Big|_{\xi \in \partial {\cal D}_{j}} = f_0(x),\right.\\
\left.  \Big( \Theta \cdot \nabla \nabla f_0 (x) - \frac{1}{\alpha_0}  \sum_{j\geq1}  \alpha_{ j} \int\limits_{ {\cal D}_{j} }  \triangle_\xi f_{j} (x, \xi) d \xi,\, \triangle_\xi f_{1}(x, \xi),\ldots,\triangle_\xi f_{N}(x, \xi)  \Big)\in L^2(E,\alpha)\right\},
\end{array}
\end{equation}
the operator $(A, \hat D(A))$ is a closed symmetric operator in $L^2(E,\alpha)$,  and $\hat D(A)$ is dense in $L^2(E, \alpha)$.

\medskip
We introduce the following two spaces:
\begin{equation}\label{h1alpha}
H^1_{ D} (E,\alpha)=
 \Big\{
f_0 \in H^1(\mathbb R^d), \  f_{j}-f_0\in L^2(\mathbb R^d; H_0^1({\cal D}_{j}))
\Big\}
\end{equation}
and
\begin{equation}\label{h1alpha}
H^2_D (E,\alpha)=
 \Big\{
f_0 \in H^2(\mathbb R^d), \  f_{j}-f_0\in L^2(\mathbb R^d; H^2({\cal D}_{j})\cap H_0^1({\cal D}_{j}))
\Big\}.
\end{equation}
Notice that
$$
 \sum\limits_{j=1}^N \alpha_{ j} \int\limits_{\mathbb R^d} \int\limits_{{\cal D}{j}} | \nabla_\xi f_{j}|^2
(x,\xi)\, d\xi dx<\infty \quad \forall F \in H^1_D (E,\alpha)
$$
and
$$\sum\limits_{j=1}^N \alpha_{ j} \|f_j\|^2_{L^2(\mathbb R^d;H^2(\mathcal{D}_j))}
<\infty \quad  \forall F \in H^2_D (E,\alpha).
$$
The space $H^{-1}(E,\alpha)$ is defined as the dual space to $H^1_D(E,\alpha)$ in $L^2(E,\alpha)$.
\begin{lemma}\label{LC}
For any $m>0$ the operator $(m-A, \hat D(A))$ is a coercive self-adjoint operator in $L^2(E, \alpha)$.
\end{lemma}
\begin{proof}
Consider the following quadratic form in $L^2(E, \alpha)$
\begin{equation}\label{Qform}
\Gamma(F,F) = \alpha_0 \int\limits_{\mathbb R^d} \Theta \nabla f_0(x) \cdot  \nabla f_0(x) \,dx +
 \sum_{j=1}^N \alpha_{ j} \int\limits_{\mathbb R^d} \int\limits_{{\cal D}{j}} | \nabla_\xi f_{j}|^2
(x,\xi)\, d\xi dx + m\|F\|^2_{L^2(E,\alpha)}
\end{equation}
with a domain $D(\Gamma)=H^1_c(E,\alpha)$.
Notice that $ f_{j}(x,\cdot)\big|_{\partial {\cal D}_{j}} = f_0 (x)$ for any $F\in D(\Gamma)$.
According to \cite[Theorem x.x]{RS} there exists a unique self-adjoint operator $\tilde A_m$  that has the following properties: \\
- its domain $D(\tilde A_m)$ is dense in $L^2(E, \alpha)$; \\
- $D(\tilde A_m)$ belongs to $D(\Gamma)$; \\
- $(\tilde A_m F,F)_{L^2(E,\alpha)}=\Gamma(F,F)$ for any $F\in D(\tilde A_m)$.

We are going to show that $\tilde A_m$ coincides with $m-A$.
First we prove that
$D(\tilde A_m)\subset\hat D(A)$. Separating the first component $f_0$ in \eqref{vectorF} we will use the notation $F=(f_0, V)$. Taking $F\in D(\tilde A_m)$ and $U=(0,U_1)\in D(\Gamma)$ with $U_1\in C_0^\infty(\mathbb R^d\,;\,
C_0^\infty({\cal D}))$, and using the relation $(\tilde A_mF,U)_{L^2(E,\alpha)}=\Gamma(F,U)$, we obtain
$$
(\tilde A_mF,U)_{L^2(E,\alpha)}=\Gamma(F,U)=\sum\limits_{j\geq 1}\alpha_j((m-\Delta_\xi)V_j,U_{1,j}),
$$
where the terms $(-\Delta_\xi V_j, U_{1,j})$ on the right-hand side are understood as a pairing between $L^2(\mathbb R^d\,;\,H^{-1}({\cal D}_j))$ and $L^2(\mathbb R^d\,;\,H^1_0({\cal D}_j))$.
This implies that $(m-\Delta_\xi)V_j\in L^2(\mathbb R^d\times {\cal D}_j)$
and $\big(0, \{(m-\Delta_\xi)V_j\}_{j\geq 1}\big)\in L^2(E,\alpha)$.
Therefore, $(0,V)\in H^2_c(E,\alpha)$.  Choosing now $U=(u_0(x),0)$ with $u_0\in C_0^\infty(\mathbb R^d)$
and considering the fact that $  \sum\limits_{j\geq1}  \alpha_{ j} \int\limits_{ {\cal D}_{j} }  \triangle_\xi V_{j} (\cdot, \xi) d \xi
\in L^2(\mathbb R^d)$,
we get $mf_0-\mathrm{div}(\Theta\nabla f_0)\in L^2(\mathbb R^d)$. Therefore, $f_0\in H^2(\mathbb R^d)$,
and $D(\tilde A_m)\subset\hat D(A)$.

Moreover, $\tilde A_mF=(m-A)F$ for any $F\in D(\tilde A_m)$.
Since $\tilde A_m$ is self-adjoint, $D(\tilde A_m)=\hat D(A)$. This yields the desired statement.
\end{proof}

We define the following set of functions:
\begin{equation}\label{CoreA}
D_A = \big\{ f_0(x) \in C_0^{\infty}(\mathbb R^d), \; f_{j} (x,\xi)-f_0(x) \in C_0^{\infty}(\mathbb R^d, {\mathbb D}_{j}(\Delta))
  \big\}.
\end{equation}
  Notice that $f_{j}(x,\xi)\big|_{\xi \in \partial {\cal D}_{j}} = f_0(x)$ for any $F=\{f_j\}_{j\geq 0}\in D_A$.

\begin{corollary}\label{cor1}
The set  $D_A \subset L^2(E,\alpha)$
defined in \eqref{CoreA} is a core of $A$, i.e.  $D_A$ is a dense subset of $L^2(E,\alpha)$ and $A = \overline{A \big |_{D_A}}$, see \cite{EK} for the details.
\end{corollary}

\begin{proof}
Clearly,  $D_A$ is a dense subset in $L^2(E,\alpha)$.
In order to show that $D_A$ is a core of $A$ we should also check that for some $m>0$ the set $\{(m-A)F,\ F\in D_A\}$  is dense
in $L^2(E,\alpha)$. Denote $J^\infty=\{(u_0,U)=(u_0(x), U_{j}(x,\xi))\,:\,u_0\in C_0^\infty(\mathbb R^d), \
U_{j}\in C_0^\infty(\mathbb R^d\,;\,C_0^\infty({\cal D}_j))\}$.  Observe that $J^\infty$ is dense in $L^2(E,\alpha)$. By Lemma \ref{LC} for an arbitrary $U\in J^\infty$ and for any $m>0$ the equation
\begin{equation}\label{cor-0}
mF-AF=U
\end{equation}
has a unique solution $F = (f_0,\, V) \in \hat D(A)$. Then the equation for $V$ can be rewritten as
\begin{equation}\label{cor-1}
(m - \Delta_\xi) (V(x,\xi)-f_0(x))=U(x,\xi)-mf_0(x)\ \ \hbox{in }\mathcal{D}, \quad (V(x,\xi)-f_0(x))\big|_{\xi \in \partial {\cal D}} = 0,
\end{equation}
or, in the coordinate form,
\begin{equation}\label{cor-2}
(m - \Delta_\xi) (V_j(x,\xi)-f_0(x))=U_j(x,\xi)-mf_0(x)\ \ \hbox{in }\mathcal{D}_j, \quad (V_j(x,\xi)-f_0(x))\big|_{\xi \in \partial {\cal D}_j} = 0,\ \ j\geq 1.
\end{equation}
From this equation we derive the following relation:
\begin{equation}\label{corrr}
V_{ j}(x,\xi)=V^I_{j}(x,\xi)+mf_0(x)V^0_{j}(\xi)+f_0(x)
\end{equation}
with $V^I_{j}=(m-\Delta_\xi)^{-1}U_j\in C_0^\infty(\mathbb R^d;\; {\mathbb D}_j(\Delta))$ and
$V^0_{j}=(m-\Delta_\xi)^{-1}1\in {\mathbb D}_j(\Delta)$.

Substituting  the right-hand side of \eqref{corrr} for $V$ into the first equation in \eqref{cor-0} yields
$$
m f_0 - \Theta \cdot \nabla\nabla f_0 - cm f_0 = w_0, 
$$
where
$$
w_0=u_0-\sum\limits_{j\geq1}\alpha_j\int\limits_{\mathcal{D}_j}\Delta_\xi(m-\Delta_\xi)^{-1}U_j(\cdot,\xi)\,d\xi,
\quad
c=-\sum\limits_{j\geq1}\alpha_j\int\limits_{\mathcal{D}_j}\Delta_\xi(m-\Delta_\xi)^{-1}1\,d\xi.
$$
Under our assumptions on $U$ we have $w_0\in C_0^\infty(\mathbb R^d)$. Also, it is straightforward to check that
$c<1$ for any $m>0$.
Consequently, $f_0$ is a Schwartz class function in $\mathbb R^d$. Taking a proper sequence of  smooth
cut-off functions $\varphi_n$   we conclude that $(m-A)(\varphi_nf_0,  V+\varphi_nf_0)$ converges in $L^2(E,\alpha)$ to $U$.
Since $(\varphi_nf_0,  V+\varphi_nf_0)\in D_A$,   this yields the desired statement.
\end{proof}

\section{The semigroup convergence.}

Applying the Hille-Yosida theorem, we conclude that $A$ is a generator of a strongly continuous, positive, contraction semigroup $T(t)$  on $L^2(E, \alpha)$.

Define a bounded linear transformation   $\pi_\eps^\omega: \, L^2(E, \alpha) \to L^2(\mathbb{R}^d)$  for every $\varepsilon \in (0,1)$ and every $\omega \in \Omega$ as follows:
\begin{equation}\label{pi_L2}
(\pi^\omega_\varepsilon F) (x) \ = \ \left\{
\begin{array}{ll}
f_0 (x), \; & \mbox{if} \; x \in \mathbb R^d \backslash G_\varepsilon^\omega; \\[2mm]
\hat f_{j}( \hat x_j^{\omega,i},\, \frac{x - \hat x_j^{\omega,i}}{\varepsilon}), \; & \mbox{if} \; x \in \varepsilon  {\cal G}^{\omega,i}_{j}, 
\end{array}
\right.
\end{equation}
where $\eps^{-1}\hat x^{\omega,i}_j$ is the vector that defines the translation which maps  $\mathcal{D}_j $ to
  $ {\cal G}^{\omega,i}_{j}$,
and
\begin{equation}\label{def_fmean}
\hat f_j(\hat x^{\omega,i}_j, \xi) = \frac{1}{\varepsilon^d | {\cal D}_{j} |} \int_{\varepsilon  {\cal D}{j}}
 f_{j}(\hat x^{\omega,i}_j + \eta, \xi ) \, d \eta.
\end{equation}

\begin{lemma}\label{pi}
Almost surely the linear operators $\pi_\varepsilon^\omega$ are uniformly bounded in the operator norm for all $\varepsilon \in (0,1)$, that
is
\begin{equation}\label{proj_bouu}
\|\pi_\varepsilon^\omega F\|_{L^2(\mathbb R^d)}\leq C \|F\|_{L^2(E, \alpha)}
\end{equation}
for any $F \in L^2(E, \alpha)$; the constant $C$ is deterministic and does not depend on $\eps$.
Moreover, for each $F \in L^2(E, \alpha)$ the following relation holds a.e.
\begin{equation}\label{conv-pi}
\| \pi^\omega_\varepsilon F ||^2_{L^2(\mathbb{R}^d)} \ \to \ \|F \|^2_{L^2(E, \alpha)} \qquad \mbox{as  } \; \varepsilon \to 0.
\end{equation}
\end{lemma}
\begin{proof}
For every $x \in \mathbb{R}^d$ and every $\omega$  we have
$$
  \sum_{j\geq 0} \chi\big._{{\cal G}^{\omega}_{j}} (x) \ = \ 1,
$$
where 
$ {\cal G}_0^\omega = {\mathbb{R}^d \backslash G^\omega_\varepsilon}$. Then we get
\begin{equation}\label{pi-2}
\begin{array}{l} \displaystyle
\| \pi^\omega_\varepsilon F \|^2_{L^2(\mathbb{R}^d)} = \int\limits_{\mathbb{R}^d} \big(\pi^\omega_\varepsilon F (x)\big)^2
 dx = \sum\limits_{j\geq 0}\ \int\limits_{\mathbb{R}^d} \big(\pi^\omega_\varepsilon F (x
 )\big)^2 \, \chi\big._{{\cal G}^{\omega}_j} (\frac{x}{\varepsilon}) dx
\\[2mm]
\displaystyle
=  \int\limits_{\mathbb{R}^d} f_0^2(x) \, \chi\big._{{\cal G}_0^\omega} (\frac{x}{\varepsilon}) dx  + \sum\limits_{j\geq1}
 \sum\limits_{i} \ \int\limits_{\mathbb{R}^d} \big(\hat f_j (\eps\hat x_{j}^{\omega,i}, \frac{x}{\varepsilon}-
 \hat x^{\omega,i}_{j})\big)^2 \,
 \chi_{{\cal G}^{\omega,i}_{j}} (\frac{x}{\varepsilon}) dx.
\end{array}
\end{equation}
By the Jenssen inequality and the definition of $\hat f_j$ in \eqref{def_fmean} this implies that
$$
\| \pi^\omega_\varepsilon F \|^2_{L^2(\mathbb{R}^d)}
\le \int\limits_{\mathbb{R}^d} f_0^2(x) \, dx  + \sum_{ j>0} \frac{1}{|{\cal D}^{j} |} \int\limits_{\mathbb{R}^d} \int\limits_{{\cal D}_{j}} \big(f_j (x, \xi)\big)^2 \, d \xi \, dx\leq \check C \|F \|_{L^2(E,\alpha)}
$$
with $\check C=\max\limits_j\{(\alpha_j^o)^{-1}\}$. This yields \eqref{proj_bouu}.

We turn to \eqref{conv-pi} and  consider the set of functions in $L^2(E,\alpha)$ which are piece-wise constant and compactly supported
with respect to the first variable $x$.  We denote this set by  ${\cal E}$ and notice that it is dense in $L^2(E,\alpha)$. If $F \in {\cal E}$ then \eqref{conv-pi} holds by the Birkhoff ergodic  theorem. 
Then, taking into account \eqref{proj_bouu} we conclude that \eqref{conv-pi} holds for any $F \in L^2(E, \alpha)$.
\end{proof}


\medskip
Now we are ready to formulate the main result of the work.

\begin{theorem}[Main theorem]\label{T1_bisL2}
For every $F \in L^2(E, \alpha)$ a.e.
\begin{equation}\label{MainTheorem}
T^\omega_\varepsilon (t) \pi^\omega_\varepsilon F \to \ T(t) F, \quad \mbox{ i.e. } \; \| T^\omega_\varepsilon (t) \pi^\omega_\varepsilon F - \pi^\omega_\varepsilon \ T(t) F\|_{L^2(\mathbb R^d)} \to 0    \quad \mbox{for all} \quad t \ge 0
\end{equation}
as $\varepsilon \to 0$.
\end{theorem}
The proof of the semigroup convergence in \eqref{MainTheorem} relies on the following approximation theorem \cite[Theorem 6.1, Ch.1]{EK}.\\[3mm]
{\bf Theorem}\, (see \cite{EK}).
{\sl \
For $n=1,2,\ldots$, let $T_n(t)$ and $T(t)$ be strongly continuous contraction semigroups on Banach space ${\cal L}_n$ and $\cal L$,
with generators $A_n$ and $A$. Let $D$ be a core for $A$.
Then the following are equivalent: \\

a) For each $f \in {\cal L} $, $T_n(t) \pi_n f \ \to \  T(t) f \; $ for all  $t \ge 0$. \\

b) For each $f \in D$, there exists $f_n \in {\cal L}_n$ for each $n \ge 1$ such that $f_n \to f$ and $A_n f_n \to Af$.
}

\medskip
According to this theorem  the semigroups convergence \eqref{MainTheorem} is a consequence of the following statement:

\begin{theorem}\label{l_Fn}
Let the generators $A$ and $A^\omega_\varepsilon$ of the strongly continuous, positive, contraction semigroups $T(t)$ and $T^\omega_\varepsilon(t)$ be defined by \eqref{A} and  \eqref{Aeps}, \eqref{a}, \eqref{DAL2}, respectively, and assume that a core
$D_A \subset L^2(E,\alpha)$ for the generator $A$ is defined by (\ref{CoreA}), and that a bounded linear transformation $\pi^\omega_\varepsilon:\ L^2(E,\alpha) \to L^2(\mathbb R^d)$ is defined by (\ref{pi_L2}) for every $\eps \in (0,1)$.

Then there exists a positive definite symmetric constant matrix $\Theta$ such that  a.s. for every $F \in D_A$, there exists $F^\omega_\varepsilon \in  D(A^\omega_\varepsilon)$ such that
\begin{equation}\label{F1}
\| F^\omega_\varepsilon - \pi^\omega_\varepsilon F\|_{L^2(\mathbb R^d)} \to 0 \quad \mbox{ and } \quad \|A^\omega_\varepsilon F^\omega_\varepsilon - \pi^\omega_\varepsilon A F\|_{L^2(\mathbb R^d)} \to 0 \quad \mbox{as } \; \varepsilon \to 0.
\end{equation}
\end{theorem}

\begin{proof}
The proof relies on the correctors technique. 
For any $F \in D_A, \ F = (f_0, \{f_{j} \})$, where
$$
f_0(x) \in C_0^{\infty}(\mathbb R^d), \quad f_{j}(x,\xi) \in C_0^{\infty}\big( \mathbb R^d; \, C^2 \big(\overline{{\cal D}_{j}} \big) \big),
$$
with
\begin{equation}\label{Feps-bis}
f_{j}(x,  \xi)\big|_{ \xi \in \partial  {\cal D}_{j} } = f_0(x), \quad x \in \mathbb R^d, \; \; \forall \  j =1, \ldots, N,
\end{equation}
we construct  the following family of functions $F_\varepsilon^\omega$ depending on the realization $\omega$ of random environment:
\begin{equation}\label{Feps}
F^\omega_\varepsilon (x) \ = \ \left\{
\begin{array}{l}
f_0 (x) + \varepsilon (\nabla f_0 (x), h^\omega_\eps (\frac{x}{\varepsilon})) +  \varepsilon^2 (\nabla \nabla f_0 (x), g_\eps^\omega(\frac{x}{\varepsilon})) + \varepsilon^2 q^\omega_\eps (x, \frac{x}{\varepsilon}), \;\; x \in \mathbb R^d \backslash G_\varepsilon^\omega, \\[2mm]
f_1 ( x,  \frac{x}{\varepsilon}) + \varepsilon \phi^\omega_1(x,  \frac{x}{\varepsilon}), \;\; x \in \varepsilon {\cal G}_1^{\omega}, \\[2mm]
\ldots \\[2mm]
f_{N}(x, \frac{x}{\varepsilon}) + \varepsilon \phi_{N}^{\omega} (x, \frac{x}{\varepsilon}), \;\; x \in \varepsilon {\cal G}^{\omega}_{N}.
\end{array}
\right.
\end{equation}
{
Here  $ h^\omega_\eps(\xi), \, g_\eps^\omega(\xi), \, q^\omega_\eps (x, \xi)$ are random functions of $ \xi$ (so-called correctors) that also depend on $\eps$;
$h^\omega_\eps(\xi)$ is the random vector function whose gradient does not depend on $\eps$, $g_\eps^\omega(\xi)$ is the random matrix function. In what follows we drop both indices $\omega$ and $\eps$ when refer to these functions.}
Correctors $\phi_{j}^{\omega} (x,\xi)$ has been introduced in order to ensure
the continuity of the function $F^\omega_\varepsilon$ and the fluxes on the boundary  $\partial\big( \varepsilon {\cal G}^{\omega}_{j} \big)$ of the corresponding inclusion.

%
%

Observe that for any $F\in D_A$ as well as for any $F \in C(E)$ and any $x \in \varepsilon {\cal G}^{\omega,i}_{j}$ we have:
\begin{equation}\label{pi-D}
(\pi^\omega_\varepsilon F) (x) = \hat f_{j}(\hat x^{\omega,i}_j, \frac{x-\hat x^{\omega,i}_{j}}{\varepsilon})= f_{j}(x, \frac{x-\hat x^{\omega,i}_{j}}{\varepsilon}) + O(\varepsilon),
\end{equation}
where the $L^\infty$ norm of $O(\eps)$ does not exceed $C\eps$.
Our goal is to choose the correctors in such a way that the function $F^\omega_\varepsilon$ defined in  \eqref{Feps} belongs to
$D(A_\eps^\omega)$,  and both relations in \eqref{F1} are fulfilled.
Denote by $B_0$ the ball in $\mathbb R^d$ centered at $0$ that contains the supports in $x$ of all the functions $f_j$, $j=0,1,\ldots, N$.

In order to introduce the correctors in  \eqref{Feps} we substitute for $F^\omega_\eps$ in the expression $A^\omega_\varepsilon F^\omega_\varepsilon - \pi^\omega_\varepsilon A F$ the right-hand side of \eqref{Feps}. Using repeatedly the formula
\begin{equation}\label{partial-xi}
\frac{\partial}{\partial x}f(x,\frac x\eps)=\Big(\frac{\partial}{\partial x}f(x,\xi)+\frac 1\eps\frac{\partial}{\partial \xi}f(x,\xi)\Big)
\Big|_{\xi=\frac x\eps},
\end{equation}
for $x \in \mathbb R^d \backslash G_\varepsilon^\omega$  after straightforward computation we obtain
\begin{equation}\label{Aout}
\begin{array}{c}
\displaystyle
(A^\omega_\varepsilon F^\omega_\varepsilon)(x) = \triangle_x \Big( f_0 (x) + \varepsilon \nabla f_0 (x)\cdot h (\frac{x}{\varepsilon}) +  \varepsilon^2 \nabla \nabla f_0 (x) \cdot g (\frac{x}{\varepsilon}) + \varepsilon^2 q (x, \frac{x}{\varepsilon}) \Big)
\\[2mm]
\displaystyle
= \Big( \triangle f_0 (x)  + 2 \nabla\nabla f_0 (x)  \nabla_{\xi} h (\xi)
 + \frac{1}{\varepsilon} \nabla f_0(x) \ \triangle_\xi h (\xi) + \nabla\nabla f_0 (x) \triangle_\xi g (\xi)
\\[2mm]
\displaystyle
 + \varepsilon^2 \triangle_x q(x, \xi) + \Xi^\omega_\eps(x,\xi)\Big)\big|_{\xi=\frac x\eps}\ ,
 \end{array}
\end{equation}
with
$$
\Xi^\omega_\eps(x,\xi)=  \Delta\nabla f_0(x)\cdot \varepsilon h (\xi)+2 \nabla\nabla\nabla f_0(x) \cdot \varepsilon \nabla_\xi g(\xi)+\Delta\nabla\nabla f_0(x) \cdot \eps^2 g (\xi)
$$
In a similar way for  $x \in  \varepsilon {\cal G}^\omega_j$ we have
\begin{equation}\label{Ain}
(A^\omega_\varepsilon F^\omega_\varepsilon)(x) = \varepsilon^2 \triangle_x \left( f_j (x,\frac{x}{\varepsilon})  + \varepsilon \phi^\omega_j (x, \frac{x}{\varepsilon}) \right) = \big(\triangle_\xi f_j (x, \xi) + \Psi^\omega_j(x,\xi)\big)\big|_{\xi=\frac x\eps}\ .
\end{equation}
with
$$
\begin{array}{l}
\displaystyle
\Psi^\omega_ j(x,\xi)
=\eps^2\Delta_x f_j(x,\xi)+2 \eps \nabla_x\cdot\nabla_\xi f_j(x,\xi)
 \\[2mm]
\displaystyle
+\eps^3 \Delta_x \phi^\omega_j(x,\xi) + 2\eps^2\nabla_x\cdot\nabla_\xi \phi^\omega_j(x,\xi) + \eps\Delta_\xi\phi^\omega_j(x,\xi).
\end{array}
$$
In order to make  $F^\omega_\varepsilon$ belong to $D(A^\omega_\varepsilon)$ we should design it in such a way that the following conditions are fulfilled on $\partial G_\eps^\omega$: \\
1) continuity condition on $\partial (\varepsilon {\cal G}^\omega_j)$
\begin{equation}\label{Bcont}
\begin{array}{cc}
\displaystyle
\Big(f_0 (x) + \varepsilon \nabla f_0 (x)\cdot h (\frac{x}{\varepsilon}) +  \varepsilon^2 \nabla \nabla f_0 (x) \cdot g(\frac{x}{\varepsilon}) + \varepsilon^2 q(x, \frac{x}{\varepsilon}) \Big) \Big|_{x\in \partial (\varepsilon {\cal G}^\omega_j)}
\\[2mm]
\displaystyle
= \Big( f_j \left( x, \frac{x}{\varepsilon} \right) + \varepsilon \phi^\omega_j (x, \frac{x}{\varepsilon}) \Big)\Big|_{x\in \partial (\varepsilon {\cal G}^\omega_j)} ;
\end{array}
\end{equation}
\\
2) continuity of fluxes condition
\begin{equation}\label{FluxC}
\begin{array}{cc}
\displaystyle
 \nabla_x \left( f_0 (x) + \varepsilon \nabla f_0 (x)\cdot h \big(\frac x\eps\big) +  \varepsilon^2 \nabla \nabla f_0 (x) \cdot g (\frac x\eps) + \varepsilon^2 q \big(x,\frac x\eps\big) \right)  \Big|_{x \in \partial( \varepsilon {\cal G}^\omega_j)} \cdot n^-
\\[3mm]
\displaystyle
=  - \varepsilon^2 \nabla_x
\Big(f_j \left( x,\frac x\eps\right) + \varepsilon \phi^\omega_j (x,\frac x\eps)\Big)  \Big|_{x \in \partial ( \varepsilon {\cal G}^\omega_j)} \cdot n^+,
\end{array}
\end{equation}
The main purpose of the functions $\phi^\omega_j(x, \frac{x}{\varepsilon})$ is to compensate the discrepancy between the inner and outer expansions for the function $F^\omega_\varepsilon$ at the boundary $\partial  {G}_\eps^\omega$, see Proposition \ref{p_existence_phi} below. It follows from \eqref{Feps-bis} that continuity condition \eqref{Bcont} leads to the relation
\begin{equation}\label{phis}
\phi^\omega_j (x, \frac{x}{\varepsilon})\Big|_{ x \in \partial (\varepsilon {\cal G}^\omega_j)} = \Big( \nabla f_0 (x) \cdot h (\frac{x}{\varepsilon}) +  \varepsilon \nabla \nabla f_0 (x) \cdot g(\frac{x}{\varepsilon}) + \varepsilon q (x, \frac{x}{\varepsilon}) \Big) \Big|_{x \in \partial (\varepsilon {\cal D}^\omega_j)}
\end{equation}
Notice that equality \eqref{phis} defines the functions $\phi^\omega_j(x,\frac{x}{\varepsilon} )$ only for  $ x \in  \partial ( \eps{\cal G}^\omega_j)$.

With the help of  \eqref{partial-xi} the  relation \eqref{FluxC} can be rewritten as
\begin{equation}\label{Bflow}
\begin{array}{cc}
\displaystyle
\Big( \nabla  f_0 (x) + \varepsilon \nabla\nabla f_0(x) h (\xi) + \nabla_\xi (\nabla f_0 (x)  h (\xi) ) + \varepsilon^2 \nabla\nabla\nabla f_0(x) g(\xi)\\[2.5mm]
\displaystyle
+  \varepsilon \nabla_\xi (\nabla\nabla f_0(x)  g(\xi)) +
\varepsilon^2 \nabla_x q (x, \xi) + \varepsilon \nabla_\xi q (x, \xi)  \Big) \Big|_{\xi = \frac{x}{\varepsilon} \in \partial \, {\cal G}^\omega_j} \cdot n^- = \\[3,5mm]
\displaystyle
- \big(\varepsilon^2 \nabla_x
f_j ( x, \xi) + \varepsilon \nabla_\xi f_j (x, \xi)  + \varepsilon^2 \nabla_\xi \phi_ j^\omega (x,  \xi) +  \varepsilon^3 \nabla_x \phi^\omega_j (x,  \xi) \big) \Big|_{\xi = \frac{x}{\varepsilon} \in \partial \, {\cal G}^\omega_j} \cdot n^+.
\end{array}
\end{equation}

We first consider the ansatz in  \eqref{Feps} in the set $\mathbb R^d \backslash G_\varepsilon^\omega$.
Collecting power-like terms in \eqref{Aout} and \eqref{Bflow}   
and considering the terms of order $\eps^{-1}$ in  \eqref{Aout} and of order $\eps^0$ in \eqref{Bflow}, we conclude that
$h(\cdot)$ should satisfy the equation
\begin{equation*}\label{hadd}
\nabla f_0(x) \ \triangle_\xi h (\xi) =0, \; \xi \in  \mathcal{G}^\omega_0
, \qquad \big( \nabla  f_0 (x) +  \nabla f_0 (x) \nabla_\xi h(\xi) \big) \cdot n^-_\xi = 0, \; \xi \in \partial \mathcal{G}_0^\omega;
\end{equation*}
here $x$ is a parameter.
Since $f_0$ does not depend on $\xi$, this problem can be rewritten as follows:
\begin{equation}\label{h}
\triangle_\xi h (\xi) =0, \;\; \xi \in \mathcal{G}^\omega_0,
 \qquad  \nabla_\xi h (\xi) \cdot n^-_\xi = - n^-_\xi, \;\; \xi \in \partial \mathcal{G}^\omega_0.
\end{equation}
This suggests the choice of $h$, it should coincide with the standard corrector used for homogenization of the Neumann problem in a random perforated domain, see \cite{JKO}.
We recall that the gradient of $h(\xi)$ is a statistically homogeneous matrix function that does not depend on $\eps$,
and $h$ satisfies equation \eqref{h}.
Moreover, $h$ shows a sublinear growth in $L^2$. Namely, assuming that $\int_{B_0}h (\frac x\eps)dx=0$, we have
\begin{equation}\label{small_norm h}
\big\| \eps h \big(\frac\cdot\eps\big)\big\|_{L^2(B_0)} \longrightarrow 0,\quad \hbox{a.s. as }\eps\to0.
\end{equation}
We also have
$$
\big\| \nabla_\xi h \big(\frac\cdot\eps\big)\big\|_{L^2(B_0)}\leq C
$$
a.s. with a constant $C$ that does not depend on $\eps$, see \cite{JKO}.

The matrix $\Theta$ in \eqref{A} is then defined by
\begin{equation}\label{theta}
\Theta \ = \  \mathbb{E} \big[ (\mathbb{I} + \nabla_\xi h (\xi)) \, \chi_{\mathcal{G}_0^\omega}(\xi) \big], \quad \mbox{i.e. } \; \Theta^{ij} \ = \   \mathbb{E} \big[ (\delta_{ij}  +  \nabla^i_\xi h^j(\xi)) \, \chi_{\mathcal{G}_0^\omega}(\xi) \big],
\end{equation}
where $\chi_{\mathcal{G}_0^\omega}(\cdot)$ is the characteristic function of ${\mathcal{G}_0^\omega}=\mathbb R^d\setminus G^\omega$.  It is proved in \cite{JKO} that $\Theta$ is positive definite.

\medskip

At the next step we collect the terms of order $\varepsilon^0$ on the right-hand side of \eqref{Aout} and equate them to
$ \Theta \cdot \nabla \nabla f_0(x)  +  \Upsilon(x)$ in order to make 
 the difference $(A^\omega_\varepsilon F^\omega_\varepsilon - \pi^\omega_\varepsilon A F ) = \big(A^\omega_\varepsilon F^\omega_\varepsilon(x)  -  \left( \Theta \cdot \nabla \nabla f_0(x)  +  \Upsilon(x) \right)\big)$
small in $L^2(\eps\mathcal{G}_0^\omega)$ norm. This yields
\begin{equation}\label{18A}
\Big(\triangle f_0 (x)  + 2 \nabla\nabla f_0 (x) \cdot \nabla_{\xi} h (\xi) +  \nabla\nabla f_0 (x) \cdot \triangle_\xi g (\xi) + \varepsilon^2 \triangle_x q (x, \xi)\Big) \big|_{\xi = \frac{x}{\varepsilon}}  =  \Theta \cdot \nabla \nabla f_0(x)  +  \Upsilon(x),
\end{equation}
where $ x \in (\eps\mathcal{G}_0^\omega \cap B_0)$;  the function $ \Upsilon(x)$ is defined in  \eqref{Y}.
We also collect the terms of order $\varepsilon^1$ in \eqref{Bflow}:
\begin{equation}\label{Bflow1}
\begin{array}{c}
\displaystyle
\varepsilon \Big( \nabla\nabla f_0(x) h (\xi) +  \nabla_\xi (\nabla\nabla f_0(x) \cdot  g (\xi)) + \nabla_\xi q (x, \xi)  \Big)  \big|_{\xi = \frac{x}{\varepsilon} \in \partial \, {\cal G}^\omega_j  } \cdot n^-
\\[2mm]
\displaystyle
= - \varepsilon  \nabla_\xi f_j (x, \xi)  \big|_{\xi = \frac{x}{\varepsilon} \in \partial \, {\cal G}^\omega_j  }  \cdot n^+.
\end{array}
\end{equation}

Selecting all the terms in \eqref{18A}-\eqref{Bflow1} that contain the second order derivatives of $f_0$, we arrive at the following problem for the random matrix valued function $g(\frac{x}{\varepsilon}) = \{ g_{ij} (\frac{x}{\varepsilon}) \}$:
\begin{equation}\label{g-omega}
\begin{array}{c}
\displaystyle
 \Big( \triangle f_0(x) + 2 \nabla\nabla f_0 (x) \cdot \nabla_{\xi} h (\xi) +  \nabla\nabla f_0 (x) \cdot \triangle_\xi g (\xi) \Big)\big|_{\xi = \frac{x}{\varepsilon}} = \Theta \cdot \nabla\nabla f_0(x),
 \\[2mm]
\displaystyle
  x \in\eps\mathcal{G}_0^\omega \cap B_0,
\\[2mm]
\displaystyle
\nabla_\xi  g(\xi) \cdot n^- \big|_{\xi = \frac{x}{\varepsilon}} = -  h (\xi) \otimes n^- \big|_{\xi = \frac{x}{\varepsilon}}, \quad x \in \eps\partial\mathcal{G}_0^\omega \cap B_0.
\end{array}
\end{equation}
In addition to these  two equations we impose the homogeneous Dirichlet boundary condition on the boundary of $B_0$
$$
g \Big(\frac x\eps\Big)=0 \quad\hbox{on } \partial B_0.
$$
Finally, $g (\frac x\eps)$ is introduced as a solution to the following problem:
  \begin{equation}\label{g-omega-V}
\begin{array}{c}
\displaystyle
\varepsilon^2 \triangle_x g (\frac{x}{\varepsilon}) \ = \ \mathbb{E}  \big[ ( \mathbf{I} +  \nabla_\xi h(\xi)) \, \chi^\omega(\xi) \big]  - \mathbf{I} 
- 2 \varepsilon \nabla_x h(\frac{x}{\varepsilon}), \quad x \in V^\varepsilon := \eps\mathcal{G}_0^\omega \cap B_0,
\\[2mm]
\displaystyle
\varepsilon \nabla_x  g (\frac{x}{\varepsilon}) \cdot n^- \ = \ -  h(\frac{x}{\varepsilon}) \otimes n^-, \quad x \in  \eps\partial\mathcal{G}_0^\omega \cap B_0,
\\[2mm]
\displaystyle
g (\frac{x}{\varepsilon}) \ = \ 0, \quad x \in \partial B_0;
\end{array}
\end{equation}
here ${\bf I}$ stands for the unit $d\times d$ matrix.
\begin{lemma}\label{l_subqua_g}
Problem \eqref{g-omega-V} has a unique solution. Moreover, a.s.
\begin{equation}\label{estim_g}
\lim\limits_{\eps\to 0} \|\eps^2 {\textstyle g\big(\frac \cdot\eps\big)}\|_{H^1(\eps\mathcal{G}_0^\omega\cap B_0)}=0.
\end{equation}
\end{lemma}
The proof of this lemma is provided in Appendix 1, Section \ref{appp_1}.
%
%

\medskip
Next, collecting the remaining terms in \eqref{18A} and \eqref{Bflow1}, we arrive at the following problem for the function $ q (x,  \frac{x}{\varepsilon})$:
\begin{equation}\label{Phi}
\begin{array}{c}
\displaystyle
\varepsilon^2 \triangle_x q (x, \frac{x}{\varepsilon}) = \Upsilon(x), \quad x \in \eps\mathcal{G}_0^\omega \cap B_0,
\\[2mm]
\displaystyle
  \nabla_\xi q (x, \xi)  \cdot n^- \big|_{\xi = \frac{x}{\varepsilon}} = - \nabla_\xi f_j (x, \xi)  \cdot n^+ \big|_{\xi = \frac{x}{\varepsilon}}, \quad x \in \eps\partial\mathcal{G}_0^\omega \cap B_0,
\end{array}
\end{equation}
where the function $\Upsilon (x) \in C_0^{\infty}(\mathbb{R}^d)$ is defined in \eqref{Y}.
We then equip system \eqref{Phi} with the homogeneous Dirichlet boundary condition at $\partial B_0$:
\begin{equation}\label{bou_Phi}
  q (x, \frac{x}{\varepsilon})=0 \quad\hbox{for }x\in\partial B_0.
\end{equation}
Denote $\Phi_\varepsilon (x) = \varepsilon^2 q (x, \frac{x}{\varepsilon})$. Let $\phi_\Phi(\cdot) \in C_0^\infty(B_0)$ be a function such that
 \begin{equation}\label{phiPhi}
\phi_\Phi\geq 0 \quad \mbox{and} \quad  \phi_\Phi=1 \hbox{ for all } x\in\{x\in\mathbb R^d\,:\,
\hbox{ there exist }j \hbox{ and } \xi \hbox{ such that }f_j(x,\xi)\not=0 \}.
\end{equation}

\begin{proposition}\label{prop_Phi}
  The following limit relations hold a.s.:
 \begin{equation}\label{addi_conv0}
  \lim\limits_{\eps\to 0} \|\Phi_\eps \|_{H^1(\eps\mathcal{G}_0^\omega \cap B_0)}=0,
  \end{equation}
  \begin{equation}\label{addi_conv1}
  \lim\limits_{\eps\to 0} \|\phi_\Phi \Phi_\eps \|_{H^1(\eps\mathcal{G}_0^\omega \cap B_0)}=0.
  \end{equation}
  Moreover,
  \begin{equation}\label{addi_conv2}
   \lim\limits_{\eps\to 0} \|\Delta_x\big(\phi_\Phi \Phi_\eps \big)-\Upsilon\|_{L^2(\eps\mathcal{G}_0^\omega\omega)}=0.
  \end{equation}
\end{proposition}
\noindent
The proof of this statement  is given  in Appendix 2.

We now turn to the correctors  $\varepsilon \phi^{\omega}_j(x,  \frac{x}{\varepsilon})$, $j=1,\ldots, N$. Our goal is to define them in such a way that
\begin{equation}\label{req1_phi}
\begin{array}{rl}
\displaystyle
   \hat f_{j}(\hat x^{\omega,i}_{j}, \frac{x - \hat x^{\omega,i}_{j}}{\varepsilon}) +\varepsilon \phi^{\omega}_j (x,  \frac{x}{\varepsilon})
=\!\!&\!\! f_0 (x) + \varepsilon (\nabla f_0 (x), h (\frac{x}{\varepsilon}))\\[2mm]
   \displaystyle &\!\!+  \varepsilon^2 (\nabla \nabla f_0 (x), g(\frac{x}{\varepsilon})) + \varepsilon^2 q (x, \frac{x}{\varepsilon})\qquad\hbox{on }\eps \partial\mathcal{G}_j^\omega,
\end{array}
\end{equation}
\begin{equation}\label{req2_phi}
\begin{array}{rl}
\displaystyle
\eps^2 \nabla\Big[ \hat f_{j}(\hat x^{\omega,i}_{j}, \frac{x -\hat x^{\omega,i}_{j\omega}}{\varepsilon}) +\varepsilon \phi^{\omega}_j (x,  \frac{x}{\varepsilon})\Big]\cdot n
\!\!&\!\! =\nabla\Big[f_0 (x) + \varepsilon (\nabla f_0 (x), h (\frac{x}{\varepsilon}))\\[2mm]
   \displaystyle &\!\!+  \varepsilon^2 (\nabla \nabla f_0 (x), g (\frac{x}{\varepsilon})) + \varepsilon^2 q (x, \frac{x}{\varepsilon})\Big]\cdot n\qquad\hbox{on } \eps \partial\mathcal{G}_j^\omega,
\end{array}
\end{equation}
\begin{equation}\label{req3_phi}
\|\varepsilon \phi^{\omega}_j(x,  \frac{x}{\varepsilon})\|_{L^2(\eps \mathcal{G}_j^\omega)}+\eps^2
\|\Delta_x\big(\varepsilon \phi^{\omega}_j (x,  \frac{x}{\varepsilon})\big)\|_{L^2(\eps \mathcal{G}_j^\omega)}\longrightarrow 0,
\quad \hbox{as }\eps\to0.
\end{equation}

\begin{proposition}\label{p_existence_phi}
  There exists a family of functions  $\phi^{\omega}_j$  with $j=1,\ldots, N$ and $\eps\in(0,1)$ such that the relations in
  \eqref{req1_phi}--\eqref{req3_phi} are fulfilled.
\end{proposition}
For the proof, see Appendix 2.
\medskip

We turn back to the {\sl Proof of Theorem}  \ref{l_Fn}.  The statement of this Theorem   is now a straightforward consequence of
\eqref{h}, \eqref{small_norm h},  Lemma \ref{l_subqua_g} and Propositions  \ref{prop_Phi} - \ref{p_existence_phi}. Indeed, due to \eqref{req1_phi} and  \eqref{req2_phi},  we have $F_\eps^\omega\in D(A_\eps^\omega)$.   Then the convergence
$$
\| F^\omega_\varepsilon - \pi^\omega_\varepsilon F\|_{L^2(\mathbb R^d)} \to 0 \quad \hbox{as }\eps\to 0
$$
follows from \eqref{pi-D}, \eqref{small_norm h}, \eqref{estim_g}, \eqref{addi_conv1} and  \eqref{req3_phi}.  Finally, by
\eqref{g-omega}, \eqref{addi_conv2} and \eqref{req3_phi} we obtain
$$
\|A^\omega_\varepsilon F^\omega_\varepsilon - \pi^\omega_\varepsilon A F\|_{L^2(\mathbb R^d)} \to 0 \quad \mbox{as } \; \varepsilon \to 0.
$$
This completes the proof of Theorem  \ref{l_Fn}.
\end{proof}

\section{Spectrum of the limit operator}

We proceed with the description of the  spectrum of the limit operator $A$ given by \eqref{A-}, and then  using the strong convergence of Markov semigroup $T^\omega_\eps (t)$ in $L^2(E)$ we  describe the limit behaviour of the spectra of operators $A^\omega_\eps$, as $\eps\to 0$ almost surely.

\medskip
Remind that each component $ f_{j} (x, \xi)$ of $F \in D_A$ can be written  as the sum
$$
f_j (x, \xi) = f_0 (x) +  r_j (x, \xi) \quad \mbox{with } \; r_j (x, \xi)\big|_{\xi \in \partial {\cal D}_j } = 0 \quad \forall \; j=1, \ldots, N.
$$
Then \eqref{A-} takes the form
\begin{equation}\label{Ag}
(-A F)(x,\xi) = \left(
\begin{array}{c}
-\Theta \cdot \nabla \nabla f_0 (x) + \frac{1}{\alpha_0}  \sum_{j=1}^{N}  \alpha_{j} \int\limits_{ {\cal D}_{j} }  \triangle_\xi r_{j} (x, \xi) d \xi \\ \\
-\triangle_\xi r_1 (x,\xi) \\  \ldots  \\
-\triangle_\xi r_N (x, \xi)
\end{array}
\right)
\end{equation}
For each $j$ the operator $- \triangle_ \xi$ on ${\cal D}_j$ with homogeneous Dirichlet boundary condition  has a discrete positive spectrum $\{ \beta^j_m \}_{m \in \mathbb N}$,   $\; \beta^j_m>0$, $\; \beta^j_m \to \infty$.
We denote by $\varkappa^j_m (\xi), \; m = 1,2, \ldots,$ the corresponding normalized eigenfunctions and by  $\mathbb M$ the set of all indices $(j,m)$.
We introduce the set $\mathbb M^* \subset \mathbb M$ of indices $(j,m)$ such that   $\int\limits_{ {\cal D}_j}  \varkappa^j_m(\xi) d \xi=\langle\varkappa^j_m\rangle\not=0$.
Let $\mathbb{B}$ be a (countable) set of all  $ \beta^j_m$:
$$
\mathbb{B} = \bigcup_{(j,m) \in \mathbb M} \beta^j_m,
$$
and
$$
b_1 = \min\limits_{(j,m) \in \mathbb M} \beta^j_m  = \min\limits_{(j,m) \in \mathbb M^*} \beta^j_m, \quad b_1>0.
$$

\begin{lemma}\label{spectrum}
The continuous spectrum $\sigma_{cont}(-A)$ of the operator $-A$ is a countable set of non-overlapping segments
$$
\sigma_{cont}(-A) \ = \  \bigcup\limits_{(j,m)\in\mathbb M^*} [ \hat\lambda^j_m, \beta^j_m],
$$
where $ \hat\lambda_1 =0$, and  $\hat\lambda^j_m  < \beta^j_m$  is the nearest to $\beta^j_m$ solution of equation
$$
\frac{1}{\alpha_0}  \sum_j \alpha_j \sum\limits_m \frac{ (u^j_m)^2 \ \beta^j_m }{\beta^j_m - \lambda}   +1 = 0 \quad \mbox{\rm with } \quad  u^j_m = \langle\varkappa^j_m\rangle.
$$
The point spectrum of the operator $-A$ is the union of eigenvalues $\beta^j_m$ with $(j,m)\in \mathbb M\setminus \mathbb M^*$:
$$
\sigma_{p}(-A) \ = \  \bigcup\limits_{(j,m)\in \mathbb M\setminus\mathbb M^*} \beta^j_m.
$$
Each eigenvalue $\beta^j_m \in \sigma_{p}(-A)$ has infinite multiplicity, so that $\sigma_{p}(-A)$ belongs to the essential spectrum of  $-A$.
\end{lemma}

\begin{proof}
Each line in the equation $-A F=\lambda F$ except of the first one reads
\begin{equation}\label{Ag2}
- A (f_0(x) + r_j (x,\xi)) = - \triangle_\xi r_j (x,\xi) = \lambda (f_0(x)+ r_j (x,\xi)), \quad  \xi \in {\cal D}_j.
\end{equation}
The function $f_0(x)$ does not depend on $\xi$, its Fourier series w.r.t. $\{\varkappa^j_m (\xi)\}$ for every $j$ takes the form
\begin{equation}\label{Ag2-f0}
f_0 (x) \cdot 1 = f_0 (x) \sum\limits_m u^j_m \varkappa^j_m (\xi), \quad \mbox{with } \; u^j_m = \int\limits_{{\cal D}_j} \varkappa^j_m (\xi) \, d\xi.
\end{equation}
Denoting by $\gamma^j_m = \gamma^j_m (x)$ the Fourier coefficients of $r_j$, from   \eqref{Ag2} - \eqref{Ag2-f0} we get
$$
- \triangle_\xi r_j (x,\xi) = \sum\limits_m \beta^j_m \gamma^j_m \varkappa^j_m (\xi) = \lambda f_0(x) \sum\limits_m u^j_m \varkappa^j_m (\xi) + \lambda \sum\limits_m \gamma^j_m \varkappa^j_m (\xi).
$$
Consequently, for any $\lambda\not\in \mathbb{B} $ we have
$
\gamma^j_m = \lambda f_0(x) \, \frac{ u^j_m}{\beta^j_m - \lambda}
$,
and thus the function
\begin{equation}\label{g1}
r_j(x, \xi) = \sum\limits_m \gamma^j_m \varkappa^j_m(\xi) =  \lambda f_0(x) \sum\limits_m \frac{ u^j_m}{\beta^j_m - \lambda} \varkappa^j_m(\xi),
\end{equation}
is a solution of equation $ - A(f_0 + r_j) = -\triangle_\xi r_j =  \lambda ( f_0 + r_j)$ for any $j$ and any $\lambda \not\in \mathbb{B} $.

Inserting  \eqref{g1} in the first line of the equation $-A F=\lambda F$ with $-A$ given by \eqref{Ag} yields
$$
-\Theta \cdot \nabla \nabla f_0 (x) - \lambda f_0(x) \frac{1}{\alpha_0}  \sum_j \alpha_j \sum\limits_m \frac{ u^j_m \ \beta^j_m }{\beta^j_m - \lambda}    \int\limits_{ {\cal D}_j }  \varkappa^j_m(\xi) d \xi = \lambda f_0(x).
$$
Consequently
\begin{equation}\label{f0}
-\Theta \cdot \nabla  \nabla f_0 (x) = \lambda f_0(x)  \left( \frac{1}{\alpha_0}  \sum_j \alpha_j \sum\limits_m \frac{ (u^j_m)^2 \ \beta^j_m }{\beta^j_m - \lambda}  + 1 \right).
\end{equation}
Since the spectrum of the operator $-\Theta \cdot \nabla  \nabla$ fills up the positive half-line, we obtain that all $\lambda>0$ such that
$$
\frac{1}{\alpha_0}  \sum_j \alpha_j \sum\limits_m \frac{ (u^j_m)^2 \ \beta^j_m }{\beta^j_m - \lambda}   +1 \ge 0
$$
belong to the spectrum of the operator $-A$.  One can easily check that the segment $[0,\beta_1]$, $b_1 = \min_{(j,m) \in \mathbb M^*} \beta^j_m>0$ belongs to the continuous spectrum of
$-A$. This implies the desired statement on $\sigma_{cont}(-A)$.

It is straightforward to check that the functions  $F^{(j,m)} = (0, \ldots, \varphi(x) \,\varkappa^j_m(\xi), 0, \ldots, 0)$ with $\varphi(x) \in L^2(\mathbb{R}^d)$ for all $(j,m)\in\mathbb M\setminus\mathbb M^*$, i.e. such that $ \langle\varkappa^j_m\rangle = 0$, are the  eigenfunctions of $-A$ with corresponding eigenvalue $\beta^j_m$.
This completes the proof.
\end{proof}

\medskip
Notice that the operators $A^\omega_\varepsilon$ for every $\varepsilon$ have statistically homogeneous coefficients, i.e. they are metrically transitive with respect to the unitary group of the space translations in $\mathbb{R}^d$. Then from the general results, see e.g. \cite{PF}, it follows that the spectra of the operators $A^\omega_\varepsilon$ are non-random for a.e. $\omega$.

\begin{proposition}
For any $\lambda \in \sigma (-A)$ a.s. there exists a sequence $\lambda_\eps$,  $\lambda_\eps\in\sigma(A^\omega_\eps)$, that converges to
$\lambda$ as $\eps\to0$.  
\end{proposition}

\begin{proof}
Since $\lambda \in \sigma (-A)$, there exist functions $F_n \in D_A$, $\| F_n \|_{L^2(E,\alpha)} =1$ such that $\| (A+\lambda) F_n \|_{L^2(E,\alpha)} \to 0$ as $n \to \infty$.
Using Theorem \ref{l_Fn}  we additionally have that for any $F_n \in D_A$ there exists $F^\omega_{n,\varepsilon} \in D(A^\omega_\varepsilon)$ for a.e. $\omega$ such that
$$
\| F^\omega_{n,\varepsilon} - \pi^\omega_\varepsilon F_n \|_{L^2(\mathbb{R}^d)} \to 0 \quad \mbox{ and } \quad \| A^\omega_\varepsilon F^\omega_{n,\varepsilon} - \pi^\omega_\varepsilon A F_n \|_{L^2(\mathbb{R}^d)} \to 0 \quad \mbox{ as } \; \varepsilon \to 0.
$$
Thus using Lemma \ref{pi} we obtain that for any (small) $\delta>0$ there exists $\varepsilon_0= \varepsilon_0(\lambda, \delta)>0$ such that for all $\varepsilon < \varepsilon_0$ there exists $ F^\omega_{n,\varepsilon} \in L^2(\mathbb{R}^d)$ with $\| F^\omega_{n,\varepsilon} \| =1$, and
\begin{equation}\label{specdelta}
\| A^\omega_\varepsilon F^\omega_{n,\varepsilon} + \lambda F^\omega_{n,\varepsilon} \|_{L^2(\mathbb{R}^d)} < \delta.
\end{equation}
This implies that there is a point of the spectrum of $-A^\omega_\eps$ in the $\delta$-neighbourhood of $\lambda$.
\end{proof}


\section{Appendix 1. The second corrector $g$. Proof of Lemma \ref{l_subqua_g}. }
\label{appp_1}

Recall the matrix valued function $g(\frac{x}{\varepsilon})$ was defined as a solution to the following problem:
%
  \begin{equation}\label{g-omegaApp2}
\begin{array}{c}
\displaystyle
\varepsilon^2 \triangle_x g (\frac{x}{\varepsilon}) \ = \ \mathbb{E}  \big[ ( \mathbf{I} +  \nabla_\xi h(\xi)) \, \chi^\omega(\xi) \big]  - \mathbf{I} 
- 2 \varepsilon \nabla_x h(\frac{x}{\varepsilon}), \quad x \in V^\varepsilon := \eps\mathcal{G}_0^\omega \cap B_0,
\\[2mm]
\displaystyle
\varepsilon \nabla_x  g (\frac{x}{\varepsilon}) \cdot n^- \ = \ -  h(\frac{x}{\varepsilon}) \otimes n^-, \quad x \in  \eps\partial\mathcal{G}_0^\omega \cap B_0,
\\[2mm]
\displaystyle
g_\eps^\omega(\frac{x}{\varepsilon}) \ = \ 0, \quad x \in \partial B_0;
\end{array}
\end{equation}
in this section we will use notation $\chi^\omega (\cdot) = \chi_{\mathcal{G}_0^\omega} (\cdot)$ for the characteristic function of the random set $\mathcal{G}_0^\omega$.
We also recall that the matrix $ \Theta $ defined by \eqref{theta} is positive definite, see \cite{JKO}.
%
In the coordinate form the above problem reads
\begin{equation}\label{g-App2}
\begin{array}{c}
\displaystyle
\varepsilon^2 \triangle_x g^{ij} (\frac{x}{\varepsilon}) \ = \ \mathbb{E}  \big[ (\delta_{ij}  +  \nabla^i_\xi h^j(\xi)) \, \chi^\omega (\xi) \big]  - \delta_{ij} 
- 2 \varepsilon \nabla^j_x h^i(\frac{x}{\varepsilon}), \quad x \in V^\varepsilon, 
\\[2mm]
\displaystyle
\varepsilon \nabla_x  g^{ij}(\frac{x}{\varepsilon}) \cdot n^- \ = \ -  h^i(\frac{x}{\varepsilon}) \, (n^-)^j, \quad x \in \partial V^\eps\cap B_0,
\\[2mm]
\displaystyle
g^{ij}(\frac{x}{\varepsilon}) \ = \ 0, \quad x \in \partial B_0.
\end{array}
\end{equation}
Each component of  $g (\frac{x}{\varepsilon}) = \{ g^{ij} (\frac{x}{\varepsilon}) \}$ can be considered separately and in what follows we omit a super index $ij$.

Denote  $\Psi_\varepsilon(\frac{x}{\varepsilon}) = \varepsilon^2 g (\frac{x}{\varepsilon})$. Our first goal is to prove that the set of functions $\Psi_\varepsilon(\frac{x}{\varepsilon})$ is bounded in $H^1(V^\varepsilon)$.  Integrating by parts and using the second equality in \eqref{g-App2} we get
\begin{equation}\label{App2-1}
\begin{array}{l}
\displaystyle
\int\limits_{V^\varepsilon} \varepsilon^2 \triangle_x g(\frac{x}{\varepsilon}) \, \varepsilon^2 g(\frac{x}{\varepsilon}) dx =
- \varepsilon^4 \int\limits_{V^\varepsilon} \big| \nabla_x g(\frac{x}{\varepsilon}) \big|^2 dx +  \varepsilon^4 \int\limits_{\partial \, V^\varepsilon} \frac{\partial g (\frac{x}{\varepsilon})}{\partial n} \, g(\frac{x}{\varepsilon}) \,  d\sigma(x)
\\[2mm]
\displaystyle
= - \varepsilon^4 \int\limits_{V^\varepsilon} \big| \nabla_x g(\frac{x}{\varepsilon}) \big|^2 dx -  \varepsilon^3 \int\limits_{\partial \, V^\varepsilon} h(\frac{x}{\varepsilon}) \, n^- \, g(\frac{x}{\varepsilon}) \, d\sigma(x).
\end{array}
\end{equation}
On the other hand, using the first equality in  \eqref{g-App2} and integrating by parts
we transform the left hand side of \eqref{App2-1} as follows:
\begin{equation}\label{App2-3}
\begin{array}{l} \displaystyle
\int\limits_{V^\varepsilon} \varepsilon^2 \triangle_x g(\frac{x}{\varepsilon}) \, \varepsilon^2 g(\frac{x}{\varepsilon}) dx
\\   [2mm]  \displaystyle
= \varepsilon^2 \int\limits_{V^\varepsilon} \mathbb{E} \big[ ( \mathbb{I} + \nabla_\xi h ) \chi^\omega \big] \, g(\frac{x}{\varepsilon}) dx -  \varepsilon^2 \int\limits_{V^\varepsilon} \big( \mathbb{I} +  \nabla_\xi h(\frac{x}{\varepsilon}) \big) g(\frac{x}{\varepsilon}) dx - \varepsilon^3 \int\limits_{V^\varepsilon}  \nabla_x h(\frac{x}{\varepsilon}) g(\frac{x}{\varepsilon}) dx
\\ [2mm]  \displaystyle
= \varepsilon^2 \int\limits_{V^\varepsilon} \left( \mathbb{E} \big[ ( \mathbb{I} + \nabla_\xi h) \chi^\omega \big] - (\mathbb{I} + \nabla_\xi h(\frac{x}{\varepsilon}) ) \right) g(\frac{x}{\varepsilon}) dx
\\ [2mm]  \displaystyle
- \varepsilon^3 \int\limits_{\partial  V^\varepsilon} h(\frac{x}{\varepsilon})  n^-  g(\frac{x}{\varepsilon})  d\sigma(x) +
\varepsilon^3 \int\limits_{V^\varepsilon}  h(\frac{x}{\varepsilon}) \nabla_x g(\frac{x}{\varepsilon}) dx.
\end{array}
\end{equation}
Thus, \eqref{App2-1} - \eqref{App2-3} imply
\begin{equation}\label{App2-4}
\varepsilon^4 \int\limits_{V^\varepsilon} \big| \nabla_x g(\frac{x}{\varepsilon}) \big|^2 dx =
- \varepsilon^2 \int\limits_{V^\varepsilon} \Big( \mathbb{E} \big[ (  \mathbb{I} + \nabla_\xi h ) \chi^\omega \big] - ( \mathbb{I} + \nabla_\xi h(\frac{x}{\varepsilon}) )\Big) g(\frac{x}{\varepsilon}) dx -
\varepsilon^3 \int\limits_{V^\varepsilon}  h(\frac{x}{\varepsilon}) \nabla_x g(\frac{x}{\varepsilon}) dx.
\end{equation}
We get from \eqref{App2-4} that
\begin{equation}\label{App2-5}
\| \nabla_x \Psi_\varepsilon \|^2_{L^2(V^\varepsilon)} \le A \| \Psi_\varepsilon \|_{L^2(V^\varepsilon)} + \| \varepsilon h \|_{L^2(V^\varepsilon)} \, \| \nabla_x \Psi_\varepsilon \|_{L^2(V^\varepsilon)}.
\end{equation}
We have used here the fact that $\Lambda^\omega_\varepsilon =  \mathbb{E} \big[ ( \mathbb{I} + \nabla_\xi h^\omega) \chi^\omega \big] - (\mathbb{I} + \nabla_\xi h^\omega(\frac{x}{\varepsilon}) ) \chi^\omega (\frac{x}{\varepsilon}) $ is a stationary random field with finite second moment $\mathbb{E} (\Lambda^\omega_\varepsilon)^2 < +\infty$. Moreover, by the Birkhoff' theorem
\begin{equation}\label{App2-5bis}
\lim_{\varepsilon \to 0}\mathbb{E} (\Lambda^\omega_\varepsilon) = 0,
\end{equation}
and thus $\Lambda^\omega_\varepsilon$ a.s. weakly converges to zero in $L^2(B_0)$ as $\eps\to0$.

\medskip
Next we apply the results on extensions in random perforated domains, see \cite{ACDP}, \cite{JKO}. According to these results
there exists a liner extension operator $L:H^1(V^\varepsilon) \to H^1(B_0)$ such that for any $f \in H^1(V^\varepsilon)$
$$
Lf|_{H^1(V^\varepsilon)} = f, \qquad \| Lf \|_{L^2(B_0)} \le C \, \|f\|_{L^2(V^\varepsilon)}, \qquad \|\nabla Lf \|_{L^2(B_0)} \le \tilde C \, \|\nabla f\|_{L^2(V^\varepsilon)},
$$
where the constants $C$ and $\tilde C$ do not depend on $\varepsilon$.  Keeping for the extended function $L\Psi_\varepsilon$
the same notation $\Psi_\varepsilon$ and considering the Dirichlet boundary condition on $\partial B_0$ in \eqref{g-App2}, by the Friedrichs inequality  we obtain
\begin{equation}\label{App2-FI}
\| \Psi_\varepsilon \|^2_{L^2(V^\varepsilon)} \le \| \Psi_\varepsilon \|^2_{L^2(B_0)}\leq
c_1 \| \nabla_x \Psi_\varepsilon \|^2_{L^2(B_0)}
\leq C \| \nabla_x \Psi_\varepsilon \|^2_{L^2(V^\varepsilon)}.
\end{equation}
Combining this with  \eqref{App2-5} yields
\begin{equation}\label{App2-6}
\| \nabla_x \Psi_\varepsilon \|_{L^2(V^\varepsilon)} \le A_1, \qquad \| \Psi_\varepsilon \|_{L^2(V^\varepsilon)} \le A_2
\end{equation}
with the constants $A_1$ and $A$ that do not depend on $\eps$.
Thus a.s.  the family of functions $\{ \Psi_\varepsilon  \}$ is bounded in $H^1(V^\varepsilon)$ and in $H^1(B_0)$. Due to the compactness of embedding  of $H^1(B_0)$ in $L^2(B_0)$ we can pass to the limit in the product $\Lambda^\omega_\varepsilon \, \Psi_\varepsilon$ as $\varepsilon \to 0$. Thus the integral
$$
\varepsilon^2 \int\limits_{V^\varepsilon} \Big( \mathbb{E} \big[ (  \mathbb{I} + \nabla_\xi h ) \chi^\omega \big] - ( \mathbb{I} + \nabla_\xi h(\frac{x}{\varepsilon}) )\Big) g(\frac{x}{\varepsilon}) dx=
\int\limits_{B_0} \Big( \mathbb{E} \big[ (  \mathbb{I} + \nabla_\xi h ) \chi^\omega \big] - ( \mathbb{I} + \nabla_\xi h(\frac{x}{\varepsilon}) )\Big)
\chi^\omega \Big(\frac x\eps\Big)\Psi_\eps (\frac{x}{\varepsilon}) dx
$$
tends to zero as $\eps\to 0$ a.s. Taking into account relations \eqref{small_norm h} we derive from \eqref{App2-4}
that
%
%
\begin{equation}\label{App2-7bis}
\| \nabla_x \Psi_{\varepsilon} \|_{L^2(B_0)} \ \to \ 0
\end{equation}
and, by the Friedrichs inequality,
\begin{equation}\label{App2-7}
\| \Psi_{\varepsilon} \|_{L^2(B_0)} \ \to \ 0.
\end{equation}

\section{Appendix 2. Proofs of  Propositions  \ref{prop_Phi} and \ref{p_existence_phi}}\label{app_3}

We begin this section by proving Proposition
\ref{prop_Phi}.  Denote  $\Phi_\varepsilon (x) : = \varepsilon^2 q (x, \frac{x}{\varepsilon})$.
Then $\Phi_\varepsilon (x) $ is a solution of the following problem:
\begin{equation}\label{App3-0}
\begin{array}{c}
\displaystyle
 \triangle_x \Phi_\varepsilon (x) \ = \ \Upsilon(x), \quad x \in V^\varepsilon =  \eps\mathcal{G}_0^\omega \cap B_0,
\\[2mm]
\displaystyle
 \nabla_x  \Phi_\varepsilon (x) \cdot n^- \ = \ - \varepsilon^2 \nabla_x f_j (x, \frac{x}{\varepsilon}) \cdot n^+, \quad x \in  \partial ( \varepsilon  \mathcal{G}_j) \cap B_0,
\\[2mm]
\displaystyle
\Phi_\varepsilon (x) \ = \ 0, \quad x \in \partial B_0.
\end{array}
\end{equation}
In what follows we will use the following notations:
$$
\nabla  f (x, \frac{x}{\varepsilon}) = \nabla_x  f (x, \frac{x}{\varepsilon}), \quad \nabla  h (\xi) = \nabla_\xi  h (\xi), \quad  \nabla \Phi_\varepsilon (x) = \nabla_x \Phi_\varepsilon (x).
$$
In order to show that the functions $\Phi_\varepsilon (x)$ are bounded in $H^1(V^\varepsilon)$ we follow the line of the proof in the previous sectiuon. Multiplying the equation in \eqref{App3-0} by   $\Phi_\varepsilon (x)$ and integrating the resulting relation over
$V^\eps$  after integration by paths we obtain
\begin{equation}\label{App3-1}
\begin{array}{l}
\displaystyle
\int\limits_{V^\varepsilon} \Upsilon(x) \,  \Phi_\varepsilon (x) dx=\int\limits_{V^\varepsilon}  \triangle \Phi_\varepsilon (x) \, \Phi_\varepsilon (x) dx =
   \int\limits_{\partial \, V^\varepsilon} \Phi_\varepsilon (x) \,  \nabla \Phi_\varepsilon (x) \cdot n^- \,  d\sigma(x)
\\[2mm]
\displaystyle
-  \int\limits_{V^\varepsilon} \big| \nabla \Phi_\varepsilon (x) \big|^2 dx=  -  \int\limits_{V^\varepsilon} \big|  \nabla \Phi_\varepsilon (x) \big|^2 dx  -  \sum\limits_j  \, \int\limits_{  \varepsilon \partial {\mathcal{G}}^\omega_j} \Phi_\varepsilon (x) \,  \varepsilon^2 \nabla f_j (x, \frac{x}{\varepsilon})  \cdot n^+ \,  d\sigma(x).
\end{array}
\end{equation}
By the Friedrichs inequality
\begin{equation}\label{App3-2bis}
\Big|\int\limits_{V^\varepsilon} \triangle \Phi_\varepsilon (x) \,  \Phi_\varepsilon (x) dx\Big| \le C_1
\|\Upsilon(x)\|_{L^2(B_0)} \, \| \nabla \Phi_\varepsilon (x) \|_{L^2(V^\varepsilon)}.
\end{equation}
with a constant $C_1$ that does not depend on $\eps$.

To estimate the second integral on the right-hand side of \eqref{App3-1} we  extend the functions $ \Phi_\varepsilon$ on $B_0$, denote the extended functions by  $\bar \Phi_\varepsilon(x)$ and apply the Stokes formula. This yields
\begin{equation}\label{App3-3}
\begin{array}{l}
\displaystyle
 \int\limits_{\varepsilon \partial {\mathcal{G}^\omega_j}} \varepsilon^2  \nabla  f_j (x, \frac{x}{\varepsilon}) \cdot n^+  \Phi_\varepsilon (x) \,  d \sigma(x)
\\ [2mm] \displaystyle
= \int\limits_{\varepsilon \mathcal{G}_j^\omega} \varepsilon^2 \triangle  f_j (x, \frac{x}{\varepsilon})  \,  \bar \Phi_\varepsilon (x) \, dx
 + \int\limits_{\varepsilon \mathcal{G}_j^\omega} \varepsilon^2  \nabla  f_j (x, \frac{x}{\varepsilon}) \cdot   \nabla \bar \Phi_\varepsilon (x) \, dx.
\end{array}
\end{equation}
From this relation by the Friedrichs inequality we derive
the following upper bound:
\begin{equation}\label{App3-4bis}
\sum\limits_j \, \Big| \int\limits_{ \varepsilon\partial \mathcal{G}_j^\omega}   \varepsilon^2 \nabla f_j (x, \frac{x}{\varepsilon})  \cdot n^+ \,  \Phi_\varepsilon (x) \, d\sigma(x) \Big| \le C \, \|   \nabla \Phi_\varepsilon   \|_{L^2(V^\varepsilon)}.
\end{equation}
Finally, \eqref{App3-1}, \eqref{App3-2bis} and \eqref{App3-4bis} imply the desired upper bound:
\begin{equation}\label{App3-5}
\|  \nabla \Phi_\varepsilon \|_{L^2(V^\varepsilon)} \le \tilde C,
\end{equation}
i.e. the  functions $\Phi_\varepsilon (x)$ are bounded in $H^1(V^\varepsilon)$.
Consequently, the extensions $\bar \Phi_\varepsilon (x)$  are also bounded in $H^1(B_0)$ and form a compact set in $L^2(B_0)$.
Thus, there exists $\Phi_0 \in H^1(B_0)$ such that, for a subsequence,
$$
\|  \bar\Phi_{\varepsilon} - \Phi_0 \|_{L^2(B_0)} \to 0.
$$
Our goal is to prove that $\Phi_0 \equiv 0$, or equivalently
\begin{equation}\label{App3-6}
\|  \Phi_\varepsilon \|_{L^2(B_0)} \to 0 \quad \mbox{as } \; \varepsilon \to 0.
\end{equation}
\medskip
For an arbitrary $\hat\psi \in C^\infty(G_\eps^\omega)$ with a compact support in $B_0$ we have
\begin{equation}\label{App3-7}
\begin{array}{c}
\displaystyle
\int\limits_{V^\varepsilon}  \triangle \Phi_\varepsilon (x) \, \hat\psi (x) \, dx
\\[2mm]
\displaystyle
= -  \int\limits_{V^\varepsilon} \nabla \Phi_\varepsilon (x) \cdot \nabla \hat\psi (x) \, dx \ + \  \sum\limits_j \ \int\limits_{\eps\partial  \mathcal{G}_j^\omega}  \varepsilon^2 \nabla f_j(x, \frac{x}{\varepsilon})  \cdot n^+ \,
\hat\psi(x) \, d\sigma(x).
\end{array}
\end{equation}
On the other hand,
\begin{equation}\label{App3-8}
\int\limits_{V^\varepsilon}  \triangle_x \Phi_\varepsilon (x) \, \hat\psi (x) dx =
\int\limits_{V^\varepsilon} \Upsilon(x) \hat\psi(x) \, dx = \int\limits_{B_0} \Upsilon(x) \, \chi_{\{\varepsilon\mathcal{G}^\omega_0\}}(x)
\, \hat\psi(x) \, dx.
\end{equation}
Therefore,
\begin{equation}\label{App3-9}
\begin{array}{l}
\displaystyle
 \int\limits_{B_0} \nabla \Phi_\varepsilon (x) \cdot \nabla \hat\psi (x) \,  \chi_{\{\eps\mathcal{G}^\omega_0\}}(x) \, dx
 \\ [2mm] \displaystyle
= \sum\limits_j \ \int\limits_{\eps\partial  \mathcal{G}_j^\omega}  \varepsilon^2 \nabla f_j (x, \frac{x}{\varepsilon})  \cdot n^+ \, 
\hat\psi(x) \, d\sigma(x) - \int\limits_{V^\varepsilon} \Upsilon(x) \hat\psi(x) \, dx.
\end{array}
\end{equation}
For an arbitrary $\psi\in C_0^\infty(B_0)$, substituting in the last relation the function
$\psi(x)+\eps h \big(\frac x\eps\big)\nabla \psi(x)$ for $\hat\psi$ we obtain
\begin{equation}\label{App3-9biss}
\begin{array}{l}
\displaystyle
 \int\limits_{B_0} \nabla \Phi_\varepsilon (x) \cdot \big(\nabla \psi (x)+
 \nabla h \big(\frac x\eps\big)\nabla \psi(x) \big) \,  \chi_{\{\eps\mathcal{G}^\omega_0\}}(x) \, dx + o(1)
 \\ [2mm] \displaystyle
= \sum\limits_j \ \int\limits_{\eps\partial  \mathcal{G}_j^\omega}  \varepsilon^2 \nabla f_j (x, \frac{x}{\varepsilon})  \cdot n^+ \, 
\hat\psi(x) \, d\sigma(x) - \int\limits_{V^\varepsilon} \Upsilon(x) \hat\psi(x) \, dx
 \\ [2mm] \displaystyle
= - \sum\limits_j \ \int\limits_{\eps \mathcal{G}_j^\omega}  \varepsilon^2 \triangle f_j (x, \frac{x}{\varepsilon})
\psi(x) \, d x - \int\limits_{V^\varepsilon} \Upsilon(x) \psi(x) \, dx + o(1),
\end{array}
\end{equation}
where $o(1)$ a.s. tends to zero as $\eps\to 0$; here we have used the inequality $\| \varepsilon h \big(\frac{\cdot}{\varepsilon}\big)\|_{H^1(\eps\mathcal{G}_0^\omega)}\leq C$ and the fact that $\| \varepsilon h \big(\frac{\cdot}{\varepsilon}\big)\|_{L^2(\eps\mathcal{G}_0^\omega)}$ vanishes as $\eps\to0$.
Using representation \eqref{Y}  of the function $\Upsilon(x)$,
the Stokes formula and the Birkhoff ergodic theorem we conclude that the right-hand side in \eqref{App3-9biss} tends to 0 as $\varepsilon \to 0$ for any $\psi \in C_0^{\infty}(B_0) $ and thus
\begin{equation}\label{App3-10}
\begin{array}{l}
\displaystyle
\lim\limits_{\eps\to0}\, \int\limits_{B_0} \nabla \Phi_\varepsilon (x) \cdot \Big(\nabla \psi (x)+
\nabla h \big(\frac x\eps\big)\nabla \psi(x) \Big)  \,  \chi_{\{\eps\mathcal{G}_0^\omega\}}(x) \, dx =0
 \\[5mm]
\displaystyle
\Phi_\varepsilon (x)\big|_{\partial B_0} =0.
\end{array}
\end{equation}
The subsequence of $ \nabla \Phi_\varepsilon$ converges weakly in $(L^2(B_0))^d$ to $\nabla\Phi_0$, as $\eps\to0$. By the definition
of  matrix $\Theta$ the sequence $ \big(\nabla \psi +
\nabla h \big(\frac \cdot\eps\big)\nabla\psi\big)  \,  \chi_{\{\eps\mathcal{G}_0^\omega\}}$ converges weakly in $(L^2(B_0))^d$
to $\Theta\nabla\psi$.  Since the function $h(\cdot)$ satisfies equation \eqref{h}, we have
$$
\mathrm{div}\big[ \big(\nabla \psi(x) +
\nabla h \big(\frac x\eps\big)\nabla\psi(x)\big)  \,  \chi_{\{\eps\mathcal{G}_0^\omega\}}(x)\big]=
\big(\Delta\psi(x)+\nabla h \big(\frac x\eps\big)\nabla\nabla\psi(x)\big)  \,  \chi_{\{\eps\mathcal{G}_0^\omega\}}(x).
$$
 The right-hand side here is bounded in $L^2(B_0)$ and thus compact in $H^{-1}(B_0)$. By the compensated compactness theorem,
 see \cite{Murat}, we obtain
$$
0=\lim\limits_{\eps\to0}\, \int\limits_{B_0} \nabla \Phi_\varepsilon (x) \cdot \Big(\nabla \psi (x)+
\nabla h \big(\frac x\eps\big)\nabla\psi\Big)  \,  \chi_{\{\eps\mathcal{G}_0^\omega\}}(x) \, dx
=\int_{B_0}\nabla\Psi_0\cdot\Theta\nabla\psi\,dx.
$$
Since $\Phi_0=0$ on $\partial B_0$, this implies that $\Phi_0=0$, and  \eqref{App3-6} follows. This convergence to $\Phi_0=0$ holds for the whole family $\Phi_\varepsilon^\omega$.

The proof of other statements of Proposition \ref{prop_Phi} is now straightforward.

\bigskip
We turn to the proof of Proposition \ref{p_existence_phi}. Denote
$$
\tilde\Xi^\eps\big(x\big)=f_0 (x)
+ \varepsilon (\nabla f_0 (x), h(\frac{x}{\varepsilon}))
+  \varepsilon^2 (\nabla \nabla f_0 (x), g(\frac{x}{\varepsilon})) + \varepsilon^2 q (x, \frac{x}{\varepsilon}).
$$
For each $(j,i)$ with $j\in\{1,\ldots,N\}$ and $i\in\mathbb Z$ we define an open set  $\eps\mathcal{Q}^\varkappa_{\omega, ji} = \{x\in\mathbb R^d\,:\, \mathrm{dist}(x,\eps\partial \mathcal{G}^{\omega,i}_{j})<\eps\varkappa \}$ and introduce in this set coordinates $y$
such that $y'=(y_2,\ldots,y_d)$ are smooth coordinates on $\eps\partial\mathcal{G}^{\omega,i}_{j}$, and $y_1$ is directed along the exterior normal,
$y_1=-\mathrm{dist}(x,\eps\partial \mathcal{G}^{\omega,i}_{j})$ if $x\in \eps\mathcal{G}^{\omega, i}_j$ and $y_1=\mathrm{dist}(x,\eps\partial \mathcal{G}^{\omega,i}_{j})$ if $x\in \mathbb R^d\setminus\eps\mathcal{G}^{\omega,i}_{j}$. Under our assumptions on the geometry of $\mathcal{G}^{\omega,i}_{j}$ there exists $\varkappa>0$ such that
\begin{itemize}
  \item $\eps\mathcal{Q}^\varkappa_{\omega,ji}$ do not intersect with $\eps\mathcal{Q}^\varkappa_{\omega, km}$, if $(j,i)\not=(k,m)$.
  \item Coordinates $y=y(x)$ are well defined in $\eps\mathcal{Q}^\varkappa_{\omega, ji}$, that is $y=y(x)$ is an invertible diffeomorphism.
\end{itemize}
Letting $\Xi^\eps\big(y\big)=\tilde\Xi^\eps\big(x(y)\big)$ and $ \check f^\eps(y)=\hat f_{j}(\hat  x^{\omega,i}_{j}, \frac{x(y)- \hat x^{\omega,i}_{j}}{\eps})$, in $\eps\mathcal{Q}^\varkappa_{\omega, ji}\cap\eps\mathcal{G}^{\omega, i}_{j}$ we then define
$$
\eps\phi^{\omega}_j(y)=\Big(\Xi^\eps\big(0,y'\big)-
 \check f^\eps_{j}(0,y')
+y_1\big[\frac\partial{\partial y_1}\Xi^\eps\big(0,y'\big)-\frac\partial{\partial y_1} \check f^\eps_{j}\big(0,y'\big)\big]-\overline\Xi\Big)
\theta\big(-\frac{y_1}{\eps}\big)+\overline\Xi,
$$
where $\theta(s)$ is a $C^\infty$ cut-off function such that $0\leq\theta\leq 1$, $\theta=1$ for $s<\frac\varkappa3$ and
 $\theta=0$ for $s>\frac{2\varkappa}3$;  $\overline\Xi$ is the mean value of $\Xi^\eps$ over
$\eps\mathcal{Q}^\varkappa_{\omega, ji}\cap(\mathbb R^d\setminus\eps\mathcal{G}^{\omega,i}_{j})$.

By construction \eqref{req1_phi} and \eqref{req2_phi} are fulfilled.  Relation \eqref{req3_phi} follows from the properties of the correctors
and elliptic estimates, see \cite[Chapter X]{GT}.

\end{document}